\documentclass[letterpaper,12pt]{amsart}
\usepackage{latexsym,array,delarray,amsthm,
amssymb,epsfig,setspace, verbatim, enumerate}
\usepackage{subfigure}

\usepackage{amssymb}
\usepackage{amsfonts}
\usepackage{amsthm}
\usepackage{graphicx}
\usepackage{color}
\usepackage{url}
\usepackage{microtype}
\usepackage{tikz}
\usepackage{hyperref}
\definecolor{darkblue}{RGB}{0,0,160}
\hypersetup{
colorlinks,%
citecolor=black,%
filecolor=black,%
linkcolor=darkblue,%
urlcolor=darkblue
}

\theoremstyle{plain}
\newtheorem{thm}{Theorem}[section]

\newtheorem{prop}[thm]{Proposition}
\newtheorem{cor}[thm]{Corollary}

\newtheorem*{thm*}{Theorem \ref{thm:main}}
\newtheorem*{lemma*}{Lemma}
\newtheorem*{prop*}{Proposition}
\newtheorem*{cor*}{Corollary}
\newtheorem*{conj*}{Conjecture}

\theoremstyle{definition}
\newtheorem{defn}[thm]{Definition}
\newtheorem{ex}[thm]{Example}

\theoremstyle{remark}
\newtheorem*{rmk}{Remark}

\newcommand{\pp}{\mathbb{P}}
\newcommand{\qq}{\mathbb{Q}}
\newcommand{\rr}{\mathbb{R}}

\newcommand{\kk}{\mathbb{K}}
\newcommand{\ff}{\mathbb{F}}

\newcommand{\cali}{\mathcal{I}}

\newcommand{\ind}{\mbox{$\perp \kern-5.5pt \perp$}}

\newcommand{\Frac}{\mathrm{Frac}}

\begin{document}

\title{Algebraic Tools for the Analysis of State Space Models}

\author{Nicolette Meshkat}
\address{Department of Mathematics and Computer Science, 
Santa Clara University, Santa Clara, CA} 
\email{nmeshkat@scu.edu}
\author{Zvi Rosen}
\address{Department of Mathematics, University of Pennsylvania, Philadelphia, PA} 
\email{zvihr@math.upenn.edu}
\author{Seth Sullivant}
\address{Department of Mathematics,  North Carolina State University, Raleigh, NC}
\email{smsulli2@ncsu.edu}


           

\begin{abstract}
We present algebraic techniques to analyze state space models
in the areas of structural identifiability, observability, and indistinguishability.  
While the emphasis is on surveying existing algebraic tools for 
studying ODE systems, we also present a variety of new results.  In particular:
On structural identifiability, we present a method using linear algebra to find identifiable functions of the parameters of a model for unidentifiable models.  On observability, we present techniques using Gr\"obner bases and algebraic matroids to test algebraic observability of state space models.  On indistinguishability, we present a sufficient condition for distinguishability using computational algebra and demonstrate testing indistinguishability.
\end{abstract}

\maketitle
\tableofcontents


\section{Introduction}

Consider a dynamic systems model in the following state space form:
\begin{equation}\label{eq:basicssm}
x'(t)=f(x(t),p,u(t)) \quad \quad \quad
y(t)=g(x(t),p)
\end{equation}
Here $x(t)$ is the state variable vector, 
$u(t)$ is the input vector (or control vector), $y(t)$ is the output vector, 
and $p$ is a parameter vector $(p_1,...,p_n)$ composed of 
unknown real parameters $p_1,...,p_n$.  In this modeling framework
the only  observed quantities are the input
 and output trajectories, $u(t)$ and $y(t)$
(or more realistically, the trajectories observed at some finite number of 
time points $t_1, t_2, \ldots$), 
together with the underlying modeling structure
(that is, the functions $f$ and $g$).
State space models are widely used throughout the applied sciences,
including the areas of control \cite{Fornasini,Pernebo,Roesser,Verhaegen}, systems biology \cite{DiStefano}, economics and finance \cite{Hamilton, Zeng}, and probability and statistics \cite{Carter, Kitagawa}.

A simple example of a state space model is a linear compartment model.

\begin{ex} \label{ex:2comp} Consider the following ODE:
$$
\begin{pmatrix} 
x_1' \\
x_2' \end{pmatrix} = {\begin{pmatrix} 
-(a_{01}+a_{21}) & a_{12} \\
a_{21} & -(a_{02}+a_{12}) 
\end{pmatrix}} {\begin{pmatrix}
x_1 \\
x_2 \end{pmatrix} } + {\begin{pmatrix}
u_1 \\
0 \end{pmatrix}}
\quad \quad y_1=x_1.$$ 
This model is called the linear 2-compartment model and will be referenced in later sections.  Here $(x_1(t),x_2(t))$ is the state variable vector, 
$u_1(t)$ is the input (or control), $y_1(t)$ is the output, 
and $(a_{01},a_{02},a_{12},a_{21})$ is the unknown parameter vector.
\end{ex}

Although the analysis of the behavior and use of state space models
falls under the dynamical systems research area umbrella,
tools from algebra can be used to analyze these models when
the functions $f$ and $g$ are rational functions.  Algebraic methods
typically focus on
determining which key features that the models satisfy  \emph{a priori}
before the models are used to analyze data.  The point of the
present paper is to give an overview of these algebraic techniques
to show how they can be applied to analyze state space models.
We focus on three main problems where algebraic techniques can be helpful:
determining structural identifiability, observability, and 
indistinguishability of the models.  We provide an overview
of techniques for these problems coming from computational algebra and 
we also introduce some new results coming from matroid theory.


\section{State Space Models}

In this section, we provide a more detailed introduction to state
space models, and the basic theoretical problems of
identifiability, observability, and indistinguishability
that we will address in this paper.  We also provide a detailed
introduction to the linear compartment models that will
be an important set of examples that we use to illustrate the
theory.

Consider a general state space model
\begin{equation}\label{eq:statespacebasic}
x'(t)=f(x(t),p,u(t)) \quad \quad \quad
y(t)=g(x(t),p)
\end{equation}
as in the introduction, with $x(t) \in \rr^N$,  $y(t) \in \rr^M$, $u(t) 
\in \rr^R$ and $p \in \rr^n$.    

The state space model (\ref{eq:statespacebasic}) is called \emph{identifiable} if 
the unknown parameter vector $p$ can be recovered from observation of
the input and output alone.  The model is \emph{observable} if
the trajectories of the state space variables $x(t)$ can be 
recovered from observation of the input and output alone.
Two state space models are \emph{indistinguishable} if for any
choice of parameters in the first model, there is a choice of parameters
in the second model that will yield the same dynamics in both models, and
vice versa.  Before getting into the technical details of these
definitions for state space models, we introduce some key examples of 
state space models that we will use to illustrate the main concepts throughout the paper.

\begin{ex}[SIR Model]\label{ex:SIR1}
A commonly used model  in epidemiology is the Susceptible-Infected-Recovered model (SIR model)  (\cite{Beretta},\cite{Bjornstad},\cite{Capasso},\cite{McCluskey},\cite{Shulgin})
which has the following form:
\begin{align*}
 S'   & = \mu N - \beta S I - \mu S  \\
 I' & = \beta S I - (\mu + \gamma) I \\
 R'& = \gamma I - \mu R  \\
 y & = k I 
\end{align*}
The interpretation of the state variables is that
$S(t)$ is the number of susceptible individuals at time $t$, 
$I(t)$ is the number of infected individuals at time $t$, and
$R(t)$ is the number of recovered individuals at time $t$.  
The unknown parameters are the birth/death rate $\mu$, 
the transmission parameter $\beta$, 
the recovery rate $\gamma$,
the total population $N$, and the proportion of the infected population measured $k$.
In this model, we assume that we only observe the trajectory $y(t)$, an (unknown) proportion of the infected population.
Note that this simple model has no input/control.

Identifiability and observability analysis in this model are concerned with determining
which unmeasured quantities can be determined from only the observed output trajectory $y$.
Identifiability specifically concerns the unobserved parameters $\mu, \beta, \gamma, N$, and $k$,
whereas observability specifically is concerned with the unobserved state variables $S$, $I$, and $R$.  \qed
\end{ex}

A commonly used family of state space models are the linear compartment models.
We outline these models here.
Let $G = (V,E) $ be a directed graph with vertex set $V$ and set of 
directed edges $E$.  Each vertex $i \in V$ corresponds to a
compartment in our model and each edge $j \rightarrow i$ corresponds to  
a direct flow of material from the $j$th compartment to the
$i$th compartment. Let
$In, Out, Leak \subseteq V$ be three sets of compartments: the
set of input compartments, output compartments, and leak compartments
respectively.  To each edge $j \rightarrow i$ we associate
an independent parameter $a_{ij}$, the rate of flow
from compartment $j$ to compartment $i$.  
To each leak node $i \in Leak$, we associate an independent
parameter $a_{0i}$, the rate of flow from compartment $i$ leaving the system.

To such a graph $G$ and set of leaks $Leak$ we associate the matrix $A(G)$ 
 in the following way:
\[
  A(G)_{ij} = \left\{ 
  \begin{array}{l l l}
    -a_{0i}-\sum_{k: i \rightarrow k \in E}{a_{ki}} & \quad \text{if $i=j$ and } i \in Leak\\
        -\sum_{k: i \rightarrow k \in E}{a_{ki}} & \quad \text{if $i=j$ and } i \notin Leak\\
    a_{ij} & \quad \text{if $j\rightarrow{i}$ is an edge of $G$}\\
    0 & \quad \text{otherwise}\\
  \end{array} \right.
\]
For brevity, we will often
 use $A$ to denote $A(G)$.
Then we construct a system of linear ODEs with inputs and outputs associated to the quadruple
$(G, In, Out, Leak)$ as follows:
\begin{equation} \label{eq:main}
x'(t)=Ax(t)+u(t)  \quad \quad y_i(t)=x_i(t)  \mbox{ for } i \in Out
\end{equation}
 where $u_{i}(t) \equiv 0$ for $i \notin In$.
 The coordinate functions $x_{i}(t)$ are the state variables, the 
 functions $y_{i}(t)$ are the output variables, and the nonzero functions $u_{i}(t)$ are
 the inputs.  The resulting model is called a   \textit{linear compartment model}.

We use the following convention for drawing linear compartment models \cite{DiStefano}.  Numbered vertices represent compartments, outgoing arrows
from the compartments represent leaks, an edge with a circle coming out of a compartment represents an output, 
and an arrowhead pointing into a compartment represents an input.  
 
\begin{figure}
\begin{center}
\resizebox{!}{3cm}{
\includegraphics{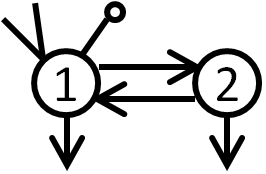}}
\end{center}\caption{A 2-compartment model with $In=\{1\}$, $Out=\{1\}$, and $Leak=\{1,2\}$.}
\label{fig:2comp}
\end{figure}

\begin{ex} \label{ex:2comp2}
For the compartment model in Figure \ref{fig:2comp}, the ODE system has the form given in Example \ref{ex:2comp}.  Since this model has a leak in every compartment, the diagonal
entries of $A(G)$ are algebraically independent of the other entries. 
In this situation,  we can re-write the diagonal entries of the matrix $A$ as $a_{11}=-(a_{01}+a_{21})$ and $a_{22}=-(a_{02}+a_{12})$.  Thus we have the following ODE system:
$$
\begin{pmatrix} 
x_1' \\
x_2' \end{pmatrix} = {\begin{pmatrix} 
a_{11} & a_{12} \\
a_{21} & a_{22} 
\end{pmatrix}} {\begin{pmatrix}
x_1 \\
x_2 \end{pmatrix} } + {\begin{pmatrix}
u_1 \\
0 \end{pmatrix}}
\quad \quad  y_1=x_1.$$ 
\end{ex}


\section{Differential Algebra Approach To Identifiability}

In this paper we focus on the  
\emph{structural} versions of identifiability, observability, and indistinguishability 
(that is, structural identifiability,
structural observability, structural indistinguishability).
That means we study when these properties hold assuming that we
are able to observe trajectories perfectly.  \emph{Practical} versions
of these problems concern how noise affects the ability to, e.g., infer
parameters of the models.  Structural answers are important because
the structural version of the condition is necessary to 
insure that the practical version holds.
On the other hand, practical versions of these problems depend on the
specific data dependent context in which the data might be observed,
and might further depend on the particular underlying unknown parameter
choices.  We will drop ``structural'' throughout the paper since this will
be implicit in the majority of our discussion.

To make the definitions of identifiability, observability, and indistinguishability
precise we will use tools from differential algebra.
In this approach, we must form the \textit{input-output equations} associated 
to our model by performing \emph{differential elimination}. 
We carry out operations in the differential ring 
\[
\qq(p)[x,y,u, x', y', u', \ldots]
\]
with the derivation $\frac{d}{dt}$ with respect to time 
such that the parameters $p$ are constants with respect to the derivation,
and $\frac{d}{dt} x = x'$, etc.  Differential algebra was
developed by Ritt \cite{Ritt1950} and Kolchin \cite{Kolchin1973} and
has its most well-known applications to the study of the algebraic solution
to systems of differential equations \cite{Singer2003}.

The goal of this differential elimination process for
state space models is to eliminate the state variables
$x(t)$ and their derivatives, so that the resulting equations are purely in terms 
of the input variables, output variables, and the parameters. 
The equations that result from applying the
differential elimination process are called the \emph{input-output equations}.
We obtain input-output equations in the following form:
\[
\sum_i{c_i(p) \psi_i (u,y)=0}
\]
where $c_i(p)$ are rational functions in the parameter vector $p$ and 
$\psi_i (u,y)$ are differential monomials in $u(t)$ and $y(t)$. 
Let $c=(c_1(p),...,c_m(p))$ denote the vector of coefficients 
of the input-output equations, which are rational functions in the parameter vector $p$.  
This coefficient vector induces a map $c:  \rr^n  \rightarrow  \rr^m$
called the coefficient map, that plays an important role in the study
of identifiability and indistinguishability.

For general state space models of the form (\ref{eq:statespacebasic})
we can also use ordinary Gr\"obner basis calculations to determine the 
input/output equation.

\begin{prop}\label{prop:grobnerio}
Consider a state space model of the form (\ref{eq:statespacebasic})
where $f$ and $g$ are polynomial functions and where there
are $N$ state-space variables, $M = 1$ output variable, and $R$ input
variables.
Let $P$ be the ideal 
\begin{multline*}
\langle x' - f(x,p,u), \, \ldots, \, 
x^{(N)} -  \frac{d^{N-1}}{dt^{N-1}}f(x,p,u), \,   
y - g(x,p), \, \ldots, \,  y^{(N)} - \frac{d^N}{dt^N}g(x,p) \rangle  \\
\subseteq \qq(p)[x,y,u, x', y', u', \ldots, x^{(N-1)}, y^{(N-1)}, u^{(N-1)}, x^{(N)}, y^{(N)}].
\end{multline*}
Then $P \cap \qq(p)[y, u, y', u', \ldots, y^{(N-1)}, u^{(N-1)}, y^{(N)}]$
is not the zero ideal and hence contains an input-output equation.
\end{prop}

Although Proposition \ref{prop:grobnerio} is known in the literature \cite{Forsman,Jirstrand},
we include a proof because it will illustrate some useful ideas that
we will use in other new results later on.  Note that although this
is stated for a single output, one can apply Proposition \ref{prop:grobnerio}
one output at a time to find input/output equations for each output
separately and hence obtain Proposition \ref{prop:grobneriomultiple}.

\begin{proof}
Note that $P$ is a prime ideal, since, with a carefully  chosen lexicographic term order,
it has as its initial ideal 
\[
\langle x', \ldots, x^{(N)}, y, \ldots, y^{(N)} \rangle
\] 
which is a prime ideal.  Since $P$ is prime, we can consider the
algebraic matroid associated to this ideal.
To say that $P \cap \qq(p)[y, u, y', u', \ldots, y^{(N-1)}, u^{(N-1)}, y^{(N)}]$
is not the zero ideal is equivalent to saying that the
set $\{y, u, y', u', \ldots, y^{(N-1)}, u^{(N-1)}, y^{(N)}\}$ is a dependent
set in the associated algebraic matroid.  The initial ideal
also shows that this ideal is a complete intersection, so it is has
codimension $N^2 + N + 1$ (since this is the number of equations involved).
The total number of variables in our polynomial ring is 
$N(N+1) + N+1 + RN$, where $N(N+1)$ counts the $x, x', \ldots$ variables, 
$N+1$ counts the $y, y', \ldots$ variables, and $RN$ counts the
$u, u', \ldots$ variables.  Thus $P$ has dimension
$N + RN$.  Since the total number of variables in the set 
$\{y, u, y', u', \ldots, y^{(N-1)}, u^{(N-1)}, y^{(N)}\}$ is $N+1 + RN$,
these variables must be dependent, i.e.~there must exist a relation.
\end{proof}

For multiple outputs, one can again take derivatives up to order $N$ and show that there must exist an input-output equation for each output:

\begin{prop}\label{prop:grobneriomultiple}
Consider a state space model of the form (\ref{eq:statespacebasic})
where $f$ and $g$ are polynomial functions and where there
are $N$ state-space variables, $M$ output variables, and $R$ input
variables.
Let $P$ be the ideal 
\begin{multline*}
\langle x' - f(x,p,u), \, \ldots, \, 
x^{(N)} -  \frac{d^{N-1}}{dt^{N-1}}f(x,p,u), \,   
y - g(x,p), \, \ldots, \,  y^{(N)} - \frac{d^N}{dt^N}g(x,p) \rangle  \\
\subseteq \qq(p)[x,y,u, x', y', u', \ldots, x^{(N-1)}, y^{(N-1)}, u^{(N-1)}, x^{(N)}, y^{(N)}].
\end{multline*}
Then $P \cap \qq(p)[y_i, u, y_i', u', \ldots, y_i^{(N-1)}, u^{(N-1)}, y_i^{(N)}]$
is not the zero ideal and hence contains an input-output equation for each $y_i$.
\end{prop}

\begin{proof}
We follow the proof of Proposition \ref{prop:grobnerio}.  The number
of equations involved is $N^2 + M(N+1)$.
The total number of variables in our polynomial ring is 
$N(N+1) + M(N+1) + RN$. Here $N(N+1)$ counts the $x, x', \ldots$ variables, 
$M(N+1)$ counts the $y, y', \ldots$ variables, and $RN$ counts the
$u, u', \ldots$ variables.  Thus $P$ has dimension
$N + RN$.  For each $y_i$, the total number of variables in the set 
$\{y_i, u, y_i', u', \ldots, y_i^{(N-1)}, u^{(N-1)}, y_i^{(N)}\}$ is $N+1 + RN$.  Thus
these variables must be dependent, i.e.~there must exist a relation for each $y_i$.
\end{proof}

Note that one could also work with smaller ideals than $P$ with only
up to $k \leq N$ derivatives, as in \cite{Meshkat2}.  In some instances this might produce an
input output equation, but the dimension guarantee that ensures the
existence of an input/output equation only occurs when $k = N$.

\begin{ex}\label{ex:SIR2}
Consider the SIR model from Example \ref{ex:SIR1}.
The ideal $P$ in this example is:
\begin{multline*}
\langle S' - \mu S - b S I + \mu N, \, 
S'' - \mu S' - \beta S I' - \beta S' I, \, 
 S'''  - \mu S'' - \beta S I'' - 2 \beta S' I' - \beta S'' I, \, \\  
 I' - (\mu+\gamma) I + \beta S I, \, 
I'' - (\mu+\gamma) I' + \beta S' I - \beta S I', \, 
 I''' - (\mu+\gamma) I'' + \beta S'' I - 2 \beta S' I' - \beta S I'', \, \\
 R' - \mu R + \gamma I, \, 
R'' - \mu R' + \gamma I', \, 
R''' - \mu R'' + \gamma I'', \, \\
y - k I, \,
y' - k I', \,
y'' - k I'', \,
y''' - k I''' \rangle .
\end{multline*}

This model has no input, so in this case we get a single output equation
in the output variable $y$ and the parameters $\mu, \beta, \gamma, N,$ and $k$.
The output equation is:
\[ 
(-\beta k N \mu + k N \mu^2 + k N \mu \gamma)y^2 + (\beta \mu + \beta \gamma)y^3 + k N \mu y y' + \beta y^2 y' - k N y'^2 + k N y y'' = 0. 
\]
This differential equation has $6$ differential monomials $y^2, y^3, y y', y^2 y', 
y'^2, y y'', $ so the coefficient vector $c$ gives a function
from $\rr^5$ to $\rr^6$, given by
\[
c:  \rr^5 \rightarrow \rr^6, \quad  ( \mu, \beta, \gamma, N,k) \mapsto
(-\beta k N \mu + k N \mu^2 + k N \mu \gamma, \beta \mu + \beta \gamma, 
k N \mu, \beta, -kN, kN). 
\]
\end{ex}

The dynamics of the input and output will only depend on the
input-output equation up to a nonzero constant multiple.  Hence, 
the coefficient map is only truly well-defined up to
scalar multiplication.  There are two natural ways to deal with this issue.
The most appealing for an algebraist is to treat
the coefficient map as a map into projective space:
$c:  \rr^n  \rightarrow  \rr\pp^{m-1}$.  The second
approach is to force the equation to have a fixed form that
will avoid this issue, by forcing the equation to be monic by
dividing through by one of the coefficients.  We will take
the second approach in this paper.  In the output equation in
Example \ref{ex:SIR2}, one possible normalization yields the coefficient map
\[
c:  \rr^5 \rightarrow \rr^6, \quad  ( \mu, \beta, \gamma, N,k) \mapsto
(-\beta  \mu +  \mu^2 +  \mu \gamma, \frac{\beta \mu + \beta \gamma}{kN}, 
\mu, \frac{\beta}{kN}, -1, 1). 
\]

In the standard differential algebra approach to identifiability,
we assume that the coefficients $c_i(p)$ of the input-output equations
can be recovered uniquely from the input-output data, and thus are 
assumed to be known quantities.  This is a reasonable assumption
when the input $u$ is a general enough function and the parameters
are generic:  in this case the dynamics will yield a unique differential
equation.  The identifiability question is 
then: can the parameters of the model be recovered 
from the coefficients of the input-output equations?

\begin{defn}\label{defn:identify}
Let $c=(c_1(p),...,c_m(p))$ denote the vector of coefficients of the input-output 
equations, which are rational functions in the parameter vector $p$,
which we assume to be normalized so that the input-output equations are monic.
We consider $c$ as a function from some natural open
biologically relevant parameter space $\Theta \subseteq \rr^n$.

\begin{itemize}
	\item  The model is \textit{globally identifiable} if $c: \Theta \rightarrow \rr^m$ 
	is a one-to-one function.
	\item  The model is   \textit{generically globally identifiable} if
	there is a dense open subset $\Theta' \subseteq \Theta$ such that
	$c: \Theta' \rightarrow \rr^m$ is one-to-one.
	\item The model is \textit{locally identifiable} if around any point 
	$p \in \Theta$ there is an open neighborhood $U_p \subseteq \Theta$ such that  
	$c : U_p  \rightarrow \rr^m$  is a one-to-one function.
	\item The model is \textit{generically locally identifiable} 
	if there is a dense open subset $\Theta' \subseteq \Theta$ such that
	for all $p \in \Theta'$ there is an open neighborhood $U_p \subseteq \Theta'$
	such that $c : U_p  \rightarrow \rr^m$  is a one-to-one function.
\item The model is \textit{unidentifiable} if there is a $p \in \Theta$ such that
 $c^{-1}(c(p))$ is infinite.
\item The model is \textit{generically unidentifiable} if there is a dense subset
$\Theta' \subseteq \Theta$ such that for all $p \in \Theta'$,  $c^{-1}(c(p))$ is infinite.
\end{itemize}
\end{defn}

As can be seen, there are many different variations on the notions of identifiability.
Because of problems that might arise on sets of measure zero that can
ruin the strongest form of global identifiability, one usually
needs to add the generic conditions to get meaningful results.
In this paper, we will consider state space models (\ref{eq:statespacebasic})
where $f$ and $g$ are polynomial (or rational) functions.  This
ensures, via the differential elimination procedure,
that the coefficient function $c(p)$ is a rational function of the
parameters.  For linear compartment models this can always be taken to be
polynomial functions.  

In this paper, we will also focus 
almost exclusively on generic local identifiability and generic
nonidentifiability and will use the 
following result to determine which of these conditions the model satisfies.  

\begin{prop} \label{prop:jacobian}
The model is generically locally
identifiable if and only if the rank of the Jacobian of $c$
is equal to $n$ when evaluated at a generic point. 
Conversely, if the rank of the Jacobian of $c$ is less
than $n$ for all choices of the parameters then the model
is generically unidentifiable.
\end{prop}

\begin{proof}
Since the coefficients in $c$ are all polynomial or rational functions of the
parameters, the model is generically locally
identifiable if and only if the image of $c$ has dimension equal to
the number of parameters, i.e.~$n$.  The dimension of the image of
a map is equal to the evaluation of the Jacobian at a generic point.
\end{proof}

\begin{ex}
{\bf SIR Model}
From Example \ref{ex:SIR2}, we have the following coefficient map:
\[
c:  \rr^5 \rightarrow \rr^6, \quad  ( \mu, \beta, \gamma, N,k) \mapsto
(-\beta  \mu +  \mu^2 +  \mu \gamma, \frac{\beta \mu + \beta \gamma}{kN}, 
\mu, \frac{\beta}{kN}, -1, 1). 
\]
We obtain the Jacobian with respect to the parameter ordering $(k,N,\mu,\gamma,\beta)$:
$$
\begin{pmatrix}
0  & 0 & -\beta + 2 \mu + \gamma  & \mu & \mu \\
\frac{-(\mu + \gamma)\beta} {k^2 N} & \frac{-(\mu + \gamma)\beta} {k N^2} & \frac{\beta} {k N} & \frac{\beta} {k N} & \frac{(\mu + \gamma)} {k N} \\
0 & 0 & 1 & 0 & 0 \\
\frac{-\beta} {k^2 N} & \frac{-\beta} {k N^2} & 0 & 0 & \frac{1} {k N} \\
0 & 0 & 0 & 0 & 0 \\
0 & 0 & 0 & 0 & 0
\end{pmatrix}.
$$
Since the rank of the Jacobian at a generic point is $4$, not $n=5$, the model is generically unidentifiable.
\end{ex} 

\subsection{Input-output equations for linear models}

There have been several methods proposed to find the input-output 
equations of nonlinear ODE models 
\cite{Audoly2001,Bearup,Evans,EvansChappell,Ljung,Meshkat2,Pohjanpalo}, 
but for linear models the problem is much simpler.  We use Cramer's rule in the following theorem, whose proof can be found in \cite{MeshkatSullivant}:

\begin{thm}\label{thm:ioeqn}
Let $\partial$ be the differential operator $d/dt$ and let 
$A_{ji}$ be the submatrix of $\partial{I}-A$ obtained by deleting the $jth$ row and the $ith$ column
of $\partial{I}-A$.  Then the input-output equations are of the form:
$$\frac{\det(\partial{I}-A)}{g_i}y_i=\sum_{j \in In} (-1)^{i + j} \frac{\det(A_{ji})}{g_i} u_j$$
where $g_i$ is the greatest common divisor of  $\det(\partial{I}-A)$, $\det(A_{ji})$ such that $j \in In$ for a given $i \in Out$.
\end{thm}

\begin{ex} \label{ex:2comp3}
{\bf Linear Compartment Model}.
 For the linear 2-compartment model from Example \ref{ex:2comp2}, we obtain the following input-output equation:
$$ y_1''-(a_{11}+a_{22})y_1'+(a_{11}a_{22}-a_{12}a_{21})y_1=u_1'-a_{22}u_1. $$
Thus we have the following coefficient map:
\[
c:  \rr^4 \rightarrow \rr^5, \quad  ( a_{11}, a_{22}, a_{12}, a_{21}) \mapsto
(1,-a_{11}-a_{22}, a_{11}a_{22}-a_{12}a_{21}, 1, -a_{22}). 
\]
We obtain the Jacobian with respect to the parameter ordering $(a_{11}, a_{22}, a_{12}, a_{21})$:
$$
\begin{pmatrix}
0 & 0 & 0 & 0 \\
-1 & -1 & 0 & 0 \\
a_{22} & a_{11} & -a_{21} & -a_{12} \\
0 & 0 & 0 & 0 \\
0 & -1 & 0 & 0
\end{pmatrix}.
$$
Since the rank of the Jacobian at a generic point is $3$, not $n=4$, the model is generically unidentifiable.
\end{ex}


\section{Identifiable functions}

One issue that arises in identifiability analysis of state space models
is figuring out what to do with a model that is generically unidentifiable.
In some circumstances, the natural approach is to develop a new model that
has fewer parameters that is identifiable.  In other circumstances, the given
model is forced upon us by the biology, and we cannot change it.  
When working with such a generically unidentifiable 
model, we would still like to determine what functions of the parameters
can be determined from given input and output data.

\begin{defn}
Let $c:  \Theta \rightarrow \rr^m$ be the coefficient map, and let
$f: \Theta \rightarrow \rr$ be another function.  We say the function $f$ is
\begin{itemize}
\item  \emph{identifiable} from $c$ if for all $p, p' \in \Theta$,
$c(p) = c(p')$ implies $f(p) = f(p')$;
\item  \emph{generically identifiable} from $c$ if there is an open dense subset
$U \subseteq \Theta$ such that $f$ is identifiable from $c$ on $U$;
\item  \emph{rationally identifiable} from $c$ if there is a rational function $\phi$
such that $\phi \circ c(p) = f(p)$ on a dense open subset $U \subseteq \Theta$;
\item  \emph{locally identifiable} from $c$ if there is an open dense subset 
$U \subseteq \Theta$ such that for all $p\in U$, there is an open neighborhood
$U_p$ such that $f$ is identifiable from $c$ on $U_p$;
\item  \emph{non-identifiable} from $c$ if there exists $p, p' \in \Theta$
such that $c(p) = c(p')$ but $f(p) \neq f(p')$; and
\item  \emph{generically non-identifiable} from $c$ if there is a subset
$U \subseteq \Theta$ of nonzero measure such that for all $p \in U$ the set 
$\{ f(p') : p' \in U \mbox{ and } c(p) = c(p') \}$ is
infinite.
\end{itemize}
\end{defn}

\begin{ex} \label{ex:2comp4} 
From the linear 2-compartment model in Example \ref{ex:2comp3}, let $p=(a_{11},a_{22},a_{12},a_{21})$ and let $c_1(p)=-a_{11}-a_{22}, c_2(p)=a_{11}a_{22}-a_{12}a_{21}, c_3(p)=-a_{22}$.  Then the functions $a_{11}, a_{22}, a_{12}a_{21}$ are rationally identifiable since
\[ a_{11} = -c_1 + c_3, \quad \quad  a_{22} = -c_3 \quad \quad   a_{12}a_{21}= -c_2-(-c_1+c_3)c_3.
\]
\end{ex}

Because we work with polynomial and rational maps $c$ and $f$ in this work,
the majority of these conditions can be phrased in algebraic language,
and checked using computer algebra.

\begin{prop}\label{prop:field} \begin{enumerate} 
\setlength{\itemindent}{-20pt}
\item The function $f(p)$ is rationally identifiable 
from $c(p) = (c_1(p),...,c_{m}(p))$ if and only if 
$ \rr(f(p),c_1(p),...,c_{m}(p)) = \rr(c_1(p),...,c_{m}(p)) $
as field extensions. \item The function $f(p)$ is locally identifiable
from  $c(p)$ if and only if $f(p)$ is algebraic 
over $\rr(c_1(p),...,c_{m}(p))$.
\item The function $f(p)$ is generically non-identifiable 
from $c(p)$ if and only if
$f(p)$ is transcendental over $\rr(c_1(p),...,c_{m}(p))$.
\end{enumerate}
\end{prop}

To explain how to use Proposition \ref{prop:field} to check the
various identifiability conditions we need to introduce some terminology. 
Associated to a set $S \subseteq \rr^m$ we have the \emph{vanishing ideal}
$\mathcal{I}(S)  \subseteq \rr[z_1, \ldots, z_m]$ defined by
\[
\mathcal{I}(S) = \langle g \in \rr[z_1, \ldots, z_m] :  g(s)  = 0 \mbox{ for all } s \in S \rangle.
\]
When $S = {\rm im}(c)$ for $c$ a rational map, the vanishing ideal can be computed
using Gr\"obner bases and elimination \cite{Cox}. Associated to the pair of
coefficient map $c$ and function $f$ that we want to test identifiability
of, we have the augmented map $\tilde{c}: \rr^n  \rightarrow \rr^{m+1}$, 
$p \mapsto (f(p), c(p))$, and
the augmented vanishing ideal 
$\mathcal{I}({\rm im}(\tilde{c})) \subseteq \rr[z_0, z_1, \ldots, z_m].$

\begin{prop}\cite[Proposition 3]{Garcia2010}\label{prop:garcia}
Suppose that $g(z_0, z) \in \mathcal{I}({\rm im}(\tilde{c})) $
is a polynomial such that $z_0$ appears in $g$ and that we may write
$g(z_0,z) = \sum_{i = 0}^dg_i(z) z_0^i$ so that $g_d(z)$ is not in
$\mathcal{I}({\rm im}(c))$.
\begin{enumerate}
\item   If $g$ is linear, $g = g_1(z) z_0 - g_0(z)$ then $f$ is rationally
identifiable from $c$ by the formula $f = \frac{g_0(c)}{g_1(c)}$.  If
in addition $g_1(z) \neq 0$ for all $z \in {\rm im}(c)$ then
$f$ is globally identifiable.
\item  If $g$ has higher degree $d > 1$ in $z_0$, then $f$ is locally identifiable,
and there are generically at most $d$ possible values for $f(p')$ among all
$p'$ with $c(p) = c(p')$.
\item  If  no such polynomial $g$ exists then $f$ is generically non-identifiable from 
$c$.
\end{enumerate}
\end{prop}

For local identifiability of a function, it is also possible to check using 
a Jacobian calculation, a result that follows easily from Proposition \ref{prop:jacobian}.

\begin{prop} \label{prop:jacident}
Let $c:\rr^{n}\rightarrow{\rr^{m}}$ be the coefficient map.
A function $f: \rr^{n} \rightarrow \rr$ is locally identifiable from $c$ if
 $\nabla f$ is in the span of the rows of the Jacobian $J(c)$.
Equivalently, consider the augmented map $\tilde{c}:  \rr^{n} \rightarrow \rr^{m+1}$.
Then $f$ is locally identifiable from $c$ if and only if
the dimension of the image of $\tilde{c}$ equals the dimension of the image
of $c$.
\end{prop}


\section{Finding identifiable functions}
\label{sec:finding}

The previous section showed how to check, given the coefficient
function $c$ and another function of the parameters $f$, whether
$f$ is identifiable from $c$ (under various variations on the definition
of identifiability).  In some circumstances, there are natural functions
to check for their identifiability (e.g.~the individual underlying parameters,
or certain specific functions with biological interpretations).  
However, when these fail to be identifiable, one would like tools to
discover new functions that are identifiable in a given state space model.
In practice the goal is to find a simple set of functions that generates
the field $\rr(c_1(p),...,c_m(p))$ (for globally identifiable functions),
or a set of functions $f_1, \ldots, f_k$ that are algebraic over $\rr(c_1(p),...,c_m(p))$
and such that $\rr(c_1(p),...,c_m(p)) \subseteq \rr(f_1(p),...,f_k(p))$
(for locally identifiable functions).  The notion ``simple'' is intentionally
left vague; typically, we mean functions of low degree that involve few parameters.
While there is no general purpose method guaranteed to solve these
problems, there are some useful heuristic approaches
that seem to work well in practice.
We highlight some of these methods in the present section.

One approach to find identifiable functions is to use Gr\"obner bases.
Specifically, one can find a Gr\"obner basis of the
ideal $\langle c_1(p)-c_1(p^*),c_2(p)-c_2(p^*),...,c_m(p)-c_m(p^*)\rangle
 \subseteq \rr(p^*)[p]$.   We state the main result from \cite{Meshkat2009}.   

\begin{prop} \cite[Theorem 1]{Meshkat2009} If $f(p)-f(p^*)$ is an element 
of a Gr\"obner basis of  $\langle c_1(p)-c_1(p^*),c_2(p)-c_2(p^*),...,
c_m(p)-c_m(p^*)\rangle$ 
for some elimination ordering of the parameter vector $p$, 
then $f(p)$ is globally identifiable.  
If instead $f(p)-f(p^*)$ is a factor of an element in the 
Gr\"obner basis of $\langle c_1(p)-c_1(p^*),c_2(p)-c_2(p^*),...,c_m(p)-c_m(p^*) \rangle$
for some elimination ordering of the parameter vector $p$, 
then $f(p)$ is locally identifiable.
\end{prop}

In practice, the Gr\"obner basis computations can be performed by picking 
a random point $p^*$ and computing a Gr\"obner basis  in the
ring $\rr[p]$.  This certifies identifiability with high probability.  
The elimination ordering is used since elements in the Gr\"obner basis at the
end of the order are likely to be sparse.

The main issue with the Gr\"obner basis approach to finding identifiable functions
is that it is unclear \textit{a priori} how many Gr\"obner bases
one needs to find in order to generate a full set of
algebraically independent identifiable functions.  
Since Gr\"obner basis computations can become computationally expensive,
we provide another approach to find identifiable functions in this paper,
using linear algebra with the Jacobian matrix $J(c)$.  
Specifically, we describe a sort of converse of  
Proposition \ref{prop:jacident}, which allows us to take
appropriate elements in the row span of $J(c)$ and deduce
that they came from an identifiable function.  
We first prove a result in the homogeneous case
and then extend to arbitrary coefficient maps via homogenization.

\begin{thm} \label{thm:homid} Let $c_i$ be a homogeneous function 
of degree $d_i$, corresponding to a coefficient of the input-output equations. 
Let $v = f_1(c) \nabla c_1 + f_2(c) \nabla c_2 + ... + f_m(c) \nabla c_m$ be 
a vector in the span of $J(c)$ over the field $\rr(c_1(p),...,c_m(p))$
(that is, each $f_i \in \rr(c_1(p),...,c_m(p))$).  
Then the dot product $v \cdot p$ is a  rationally identifiable function.
If each $f_i$ is locally identifiable then $v \cdot p$
is locally identifiable.

\end{thm}

To prove Theorem \ref{thm:homid} we make use of the 
Euler homogeneous function theorem.

\begin{prop}[Euler's Homogeneous Function Theorem]  
Let $f$ be a homogeneous function of degree $d$.  
Then $ f = \dfrac{1}{d} \displaystyle
\sum_i{p_i \dfrac{\partial{f}}{\partial{p_i}}}$.  
\end{prop}

\begin{proof}[Proof of Theorem \ref{thm:homid}]
Let $f = (f_1, \ldots, f_m)$ be the row vector of the $f_i's$.  
The function $v \cdot p$ has the form
\[
v \cdot p  =   f J(c) p.
\]
The rows of $J(c)$ are the gradients of the $c_i$'s.  Since these functions are
homogeneous, we have that $J(c)p =  (d_1 c_1(p), \ldots, d_m c_m(p))^T$
by Euler's homogeneous function theorem.  But then
\[
v \cdot p  =   f J(c) p = f (d_1 c_1(p), \ldots, d_m c_m(p))^T  =
\sum_{i = 1}^m  f_i(p) d_i c_i(p)
\]
which expresses $v \cdot p$ as a polynomial function in elements of $\rr(c_1(p),...,c_m(p))$,
so $v \cdot p$ is rationally identifiable.  If each $f_i$ were locally
identifiable, $v \cdot p$ would belong to an algebraic extension of 
$\rr(c_1(p),...,c_m(p))$ and hence be locally identifiable.
\end{proof}

Theorem \ref{thm:homid} must be used in conjunction with 
Gaussian elimination and Proposition \ref{prop:garcia} or \ref{prop:jacident}.
Indeed, our strategy in implementations is to attempt
Gaussian elimination cancellations starting with the Jacobian
matrix $J(c)$.  At each step when we want to perform an
elementary operation, we use Proposition \ref{prop:garcia} or \ref{prop:jacident}
to check whether the corresponding multiplier is rationally identifiable or
locally identifiable.  An approach based completely on linear algebra
would only make use of Proposition \ref{prop:jacident} in which case we
only deduce local identifiability.

\begin{ex} \label{ex:2comp5} Let $c$ be the map $p \mapsto (c_1(p), c_2(p), c_3(p))$ from the linear 2-compartment model in Example \ref{ex:2comp4}.  Then the Jacobian $J(c)$ is given by
$$
\begin{pmatrix}
-1 & -1 & 0 & 0 \\
a_{22} & a_{11} & -a_{21} & -a_{12} \\
0 & -1 & 0 & 0
\end{pmatrix}.
$$
Then applying Gaussian elimination over $\rr(c_1(p),c_2(p),c_3(p))$, we obtain:
$$
\begin{pmatrix}
-1 & 0 & 0 & 0 \\
0 & -1 & 0 & 0 \\
0 & 0 & -a_{21} & -a_{12}
\end{pmatrix}.
$$
This implies that $-a_{11}, -a_{22}$ and $-2a_{12}a_{21}$ are all locally identifiable.  Thus, $a_{11}, a_{22}$ and $a_{12}a_{21}$ are locally identifiable.
\end{ex}

\begin{rmk} Note that in Example \ref{ex:2comp4}, we obtained that the functions $a_{11}, a_{22}$ and $a_{12}a_{21}$ are rationally identifiable, whereas in Example \ref{ex:2comp5}, we only obtained that the functions $a_{11}, a_{22}$ and $a_{12}a_{21}$ are locally identifiable.  This is the cost of not using a Gr\"obner basis.  
\end{rmk}  

\begin{rmk} The identifiable functions obtained using linear algebra on 
the Jacobian matrix depend heavily on the specific column ordering of the 
Jacobian matrix chosen.  Thus, for a given column ordering (corresponding 
to a given parameter ordering), we may not generate the ``simplest''  
locally identifiable functions.  We do, however, always generate identifiable 
functions, as opposed to the Gr\"obner basis approach, in which there is no 
guarantee of generating elements/factors of elements of the form $f(p)-f(p^*)$ for 
a given elimination ordering $p$.
\end{rmk}

\begin{ex} From the SIR Model in Example \ref{ex:SIR2}, we can form the following coefficient map, ignoring constant coefficients:
$$
c(k,N,\mu,\gamma,\beta)=(-\beta \mu + \mu^2 + \mu \gamma, \frac{(\mu + \gamma)\beta} {k N}, \mu, \frac{\beta} {k N})
$$
thus we obtain the following Jacobian with respect to the parameter ordering $(k,N,\mu,\gamma,\beta)$:
$$
\begin{pmatrix}
0  & 0 & -\beta + 2 \mu + \gamma  & \mu & \mu \\
\frac{-(\mu + \gamma)\beta} {k^2 N} & \frac{-(\mu + \gamma)\beta} {k N^2} & \frac{\beta} {k N} & \frac{\beta} {k N} & \frac{(\mu + \gamma)} {k N} \\
0 & 0 & 1 & 0 & 0 \\
\frac{-\beta} {k^2 N} & \frac{-\beta} {k N^2} & 0 & 0 & \frac{1} {k N} \\
\end{pmatrix}
$$
from this we get the row-reduced Jacobian:
$$
\begin{pmatrix}
N  & k & 0  & 0 & 0 \\
0 & 0 & 1 & 0 & 0 \\
0 & 0 & 0 & 1 & 0 \\
0 & 0 & 0 & 0 & 1 
\end{pmatrix}.
$$
Thus, dotting each row vector with $p$ and dividing each polynomial by their respective degrees,
we find that $k N, \mu, \gamma, \beta $ are 
locally identifiable.  
\end{ex}


When the coefficient functions $c_i(p)$ are not homogeneous functions, 
we can homogenize the functions by some variable $z$ and 
add $z$ to the list $c$ of identifiable functions.  This results
in a similar identifiability result.

\begin{thm} \label{thm:nonhomid} Let $\tilde{c_i}$ be the homogenization of the coefficient
function $c_i$ and suppose it has degree $d_i$. 
Let $v = f_1(\tilde{c},z) \nabla \tilde{c_1} + f_2(\tilde{c},z) \nabla \tilde{c_2} 
+ ... + f_m(\tilde{c},z) \nabla \tilde{c_m}$ 
be a vector in the span of $J(\tilde{c},z)$ 
over the field $\rr(\tilde{c_1}(p,z),...,\tilde{c_m}(p,z),z)$.  
Then the dot product ${v \cdot (p,z) | }_{z=1}$ is a rationally identifiable function.
If $f_1, \ldots, f_m$ are locally identifiable  given $c$ then 
${v \cdot (p,z) | }_{z=1}$ is locally identifiable. 
\end{thm}

\begin{proof} 
Clearly $v \cdot (p,z)$ is rationally identifiable over the field 
$\rr(\tilde{c_1}(p,z),...,\tilde{c_m}(p,z),z)$ by Theorem \ref{thm:homid}.  
We need to show that setting $z=1$ preserves identifiability.  
Since $f(p,z)=v \cdot (p,z)$ is algebraic over the field 
$\rr(\tilde{c_1}(p,z),...,\tilde{c_m}(p,z),z)$, then clearly 
$f(p,z)|_{z=1}$ is algebraic over the field $\rr(\tilde{c_1}(p,z)|_{z=1},...,
\tilde{c_m}(p,z)|_{z=1},z|_{z=1})$.  
Since $\tilde{c_i}(p,z)|_{z=1}$ is precisely $c_i(p)$, 
then $f(p,z)|_{z=1}$ is in the field $\rr(c_1(p),...,c_m(p))$.  If $f_1, \ldots, f_m$
are algebraic over $\rr(\tilde{c_1}(p,z),...,\tilde{c_m}(p,z),z)$ then 
${v \cdot (p,z) | }_{z=1}$ is algebraic over $\rr(c_1(p),...,c_m(p))$.
\end{proof}

\begin{ex} Let $c$ be the map $(p_1,p_2,p_3) \mapsto (p_{1}^2, p_{1}^2+p_{1}p_{3}+p_{1}p_{2}^{2}p_{3})$.  Then the homogenized map $\tilde{c}$ is the map $(p_1,p_2,p_3,z) \mapsto (p_{1}^2, p_{1}^2 z^2+p_{1}p_{3} z^2+p_{1}p_{2}^{2}p_{3})$.  Then the Jacobian $J(\tilde{c},z)$ is given by
$$
\begin{pmatrix}
2p_1 & 0 & 0 & 0\\
2p_{1} z^2 + p_{3} z^2 + p_{2}^2p_{3} & 2p_{1}p_{2}p_{3} & p_{1} z^2+ p_{1}p_{2}^2 & 2p_{1}^2 z + 2p_{1}p_{3}z\\
0 & 0 & 0 & 1
\end{pmatrix}.
$$
Then applying Gaussian elimination over $\rr(\tilde{c_1}(p,z),\tilde{c_2}(p,z),z)$, we obtain:
$$
\begin{pmatrix}
1 & 0 & 0 & 0\\
0 & 2p_{2}p_{3} & z^2+p_{2}^2 & 2(p_{1} + p_{3})z\\
0 & 0 & 0 & 1
\end{pmatrix}.
$$
Thus, dotting each row vector with $(p,z)$, we obtain $p_{1}, 3p_{2}^2p_{3}+2p_{1}z^2+3p_{3}z^2$, and $z$ are locally identifiable.  Dividing by the degree and setting $z=1$, we obtain that $p_{1}$ and $p_{2}^2p_{3}+2p_{1}/3+p_{3}$ are locally identifiable.
\end{ex}


\section{Observability}

In this section we explore how algebraic and combinatorial tools
can be used to determine whether or not the state variables are observable.
Roughly speaking, the state variable $x_i$  is \emph{observable} if 
it can be recovered from observation of the input and output alone. 
We will use algebraic language to make this precise and explain
how Gr\"obner bases and matroids can be used to check this condition.

\begin{defn}
Consider a state space model of form (\ref{eq:statespacebasic}).
\begin{itemize}
\item  The state variable $x_i$ is 
\emph{generically observable} given the input and
output trajectories and generic parameter value $p$
if there is a unique trajectory for $x_i$ compatible
with the given input/output trajectory.
\item  The state variable $x_i$ is \emph{rationally observable} 
given input and output trajectories and generic parameter 
value $p$ if there is a rational function $F$ 
such that the trajectory  $x_i(t)$ satisfies 
$x_i(t)  =  F(y, y', \ldots, u, u', \ldots, p)$.   
\item  The state variable $x_i$ is \emph{generically locally observable} if 
given a generic parameter vector, there is an open neighborhood $U_{x_i}$ of the trajectory
$x_i(t)$ such that there is no other trajectory $\tilde{x}_i(t) \subseteq U_{x_i}$
that is compatible with input/output data.
\item  The state variable $x_i$ is \emph{generically unobservable} if
given the input and output trajectories and a generic parameter value $p$
there are infinitely many trajectories for $x_i$ compatible with the
given input/output trajectory.
\end{itemize}
\end{defn}

As usual, when $f$ and $g$ are polynomial functions, we can give equivalent definitions
to many of these conditions, and algebraic methods for checking them.

The following proposition gives algebraic conditions for observability.  More details on the differential algebra involved can be found in \cite{Glad}.

\begin{prop}\label{prop:diffalgobs}
Consider a state space model of form (\ref{eq:statespacebasic}).
Let $\Pi$ be the differential ideal generated the polynomials
$x' - f(x,p,u),  y - g(x,p)$.  Let $h \in 
\Pi \cap \qq(p)[x_i, y, y', \ldots, u, u', \ldots]$
be a polynomial and write this as $h =  \sum_{j = 0}^k h_j  x_i^j$
where each $h_j \in \qq(p)[ y, y', \ldots, u, u', \ldots]$, $k \geq 1$, and
$h_k  \notin \Pi$.  Then
\begin{itemize}
\item  If $k = 1$, then $x_i$ is rationally observable.
\item  If $k > 1$, then $x_i$ is locally observable.
\item  If there is no polynomial $h \in 
\Pi \cap \qq(p)[x_i, y, y', \ldots, u, u', \ldots]$
satisfying the three conditions then $x_i$ is generically unobservable.
\end{itemize}
\end{prop}

As with computations for finding the input/output equations, one
does not need to explicitly use the differential algebra to check the
conditions of Proposition \ref{prop:diffalgobs}, and it is possible
to do this directly via Gr\"obner bases and properties of the 
Jacobian matrix.

\begin{prop}\label{prop:observablegrobner}
Consider a state space model of the form (\ref{eq:statespacebasic})
where $f$ and $g$ are polynomial functions and where there
are $N$ state-space variables, $M = 1$ output variable, and $R$ input
variables.
Let $P$ be the ideal 
\begin{multline*}
\langle x' - f(x,p,u), \, \ldots, \, 
x^{(N-1)} -  \frac{d^{N-2}}{dt^{N-2}}f(x,p,u), \,   
y - g(x,p), \, \ldots, \,  y^{(N-1)} - \frac{d^{N-1}}{dt^{N-1}}g(x,p) \rangle  \\
\subseteq \qq(p)[x,y,u, x', y', u', \ldots, x^{(N-2)}, y^{(N-2)}, u^{(N-2)}, x^{(N-1)}, y^{(N-1)}].
\end{multline*}
Consider an elimination ordering $<$  on 
$\qq(p)[x,y,u, x', y', u', \ldots, x^{(N-2)}, y^{(N-2)}, u^{(N-2)}, x^{(N-1)}, y^{(N-1)}]$
with three blocks of variables
\[
\{x, x', \ldots, x^{(N-1)}\}  \setminus \{x_i\}   \quad >  \quad   \{x_i\} \quad
>  \quad \{y,u, y', u', \ldots, y^{(N-2)}, u^{(N-2)}, y^{(N-1)} \}.
\]
Then a Gr\"obner basis for $P$ with respect to $<$ will contain
a polynomial of the type indicated in Proposition \ref{prop:diffalgobs}
if it exists. Otherwise no such polynomial exists. 
\end{prop}

The proof of Proposition \ref{prop:observablegrobner} is a combination 
of the ideas of Propositions \ref{prop:grobnerio} and \ref{prop:garcia}.

\begin{proof}
First we need to show that we can find such a polynomial, if it exists,
only looking up to derivatives of order $N-1$.  This follows a similar
argument as the proof of Proposition \ref{prop:grobnerio} by dimension counting.
The codimension of $P$ is $N(N-1)+N$, the total number of variables in our polynomial
ring is $N^2+N+R(N-1)$, and thus $P$ has dimension $N+R(N-1)$.
Since the total number of variables in the set 
$\{x_i,y, u, y', u', \ldots, y^{(N-2)}, u^{(N-2)}, y^{(N-1)}\}$
is $1+N+R(N-1)$, these variables must be dependent, i.e.~there must exist a relation.  
If all the relations that exist do not involve $x_i$ in a nontrivial way,
there will not exist such relations if we add more derivatives.
Indeed, adding one more set of derivatives then there must exist an
input-output equation involving the variable $y^{(N)}$ and lower order
terms in $y$, by the proof of Proposition 
\ref{prop:grobnerio}.  Hence these could be used to eliminate any appearance
of $y^{(N)}$ or higher in any putative constraint involving $x_i$.  
Since the only equation in our system that involves $y^{(N)}$ was the equation
$y^{(N)} - \frac{d^{N}}{dt^{N}}g(x,p)$, this means we need not have added it
to our system since it cannot be eliminated by interacting with other equations.
However, without this equation, there is only a single appearance of $x^{(N)}$,
so there is no way to eliminate those variables that involves using those equations,
and hence we are reduced to our system just up to order $N-1$.

Now we will show that the Gr\"obner basis computation produces the desired equation.
Suppose there is an equation $h$ of the desired type in the ideal $P$.
If the Gr\"obner basis of $P$ did not contain a polynomial of the desired type,
then the Gr\"obner basis of $P$ does not contain a polynomial
in the variables $\{x_i, y,u, y', u', \ldots, y^{(N-2)}, u^{(N-2)}, y^{(N-1)} \}$
that involves the variable $x_i$.
Then reducing $h$ by the Gr\"obner basis cannot produce the zero polynomial,
contradicting that we had a Gr\"obner basis.
\end{proof}

Proposition \ref{prop:observablegrobner} can be generalized
to situations where there is more than one output variable.  Indeed, from Proposition \ref{prop:grobneriomultiple}, we can obtain input-output equations for each $y_i$.  Following a similar dimension counting argument, we obtain that $P$ has dimension $N+R(N-1)$ and the total number of variables in the set $\{x_i,y, u, y', u', \ldots, y^{(N-2)}, u^{(N-2)}, y^{(N-1)}\}$ is $1+MN+R(N-1)$, thus these variables must be dependent, i.e.~there must exist a relation.
In this case, one might be able to get away with looking at derivatives
of lower orders in some of the variables (i.e.~not all the way to $N-1$)
however this will depend on the structure of the underlying system.
Making this precise depends on terminology from differential algebra
that we would like to avoid.  See \cite{Glad} for details.  
One typical corollary is the following.

\begin{cor}\label{cor:orderobserve}
Consider a state space model of the form (\ref{eq:statespacebasic})
where $f$ and $g$ are polynomial functions and where there
are $N$ state-space variables, $M = 1$ output variable, and $R$ input
variables.  If the input/output equation has order $N$, then
all the state space variables are locally observable.
\end{cor}

\begin{proof}
The proof of Proposition \ref{prop:observablegrobner} shows that 
after adding the $N-1$ derivatives, there must exist a relation among
the set 
$\{x_i,y, u, y', u', \ldots, y^{(N-2)}, u^{(N-2)}, y^{(N-1)}\}$.
However, this could not be just among the set of variables 
$\{y, u, y', u', \ldots, y^{(N-2)}, u^{(N-2)}, y^{(N-1)}\}$ since this would be an
input/output equation of order $<N$.
\end{proof}

\begin{ex} 
From Example \ref{ex:2comp2}, let our model be of the form:
$$
\begin{pmatrix} 
x_1' \\
x_2' \end{pmatrix} = {\begin{pmatrix} 
a_{11} & a_{12} \\
a_{21} & a_{22} 
\end{pmatrix}} {\begin{pmatrix}
x_1 \\
x_2 \end{pmatrix} } + {\begin{pmatrix}
u_1 \\
0 \end{pmatrix}},
\quad \quad  y=x_1.$$ 
Taking derivatives, we have the system of equations: \[
\langle a_{11}x_1+a_{12}x_2+u_1-x_1', \,  a_{21}x_1+a_{22}x_2-x_2', \,  x_1-y, \,  x_1'-y' \rangle.
\]
These are polynomials in the polynomial ring 
$\rr(p)[x_1,x_2,x_1',x_2',u_1,y,y']$.

Using the elimination order specified to calculate a Gr\"obner basis, 
we see that $a_{11}y_1+a_{12}x_2+u_1-y'$ and $x_1-y$ are two polynomials 
of the desired form.  Thus the model is rationally observable.
Alternatively, the input-output equation for this model is of 
differential order $2$, which equals the number of state variables,
so the model is locally observable by Corollary \ref{cor:orderobserve}.
\end{ex}

The main problem with this definition of observability is 
that appears to require explicit computation of the desired polynomials. 
However, instead of applying a Gr\"obner 
basis to find the desired polynomials, 
we can examine the algebraic matroid associated to this system.

The algebraic matroid is equivalent to the linear matroid
of differentials, for which computations are much simpler.
Because the definition of observability distinguishes between a
variable and its derivatives, we also treat them separately in our
discussion. The ground set of the 
matroid for an observability computation is
\[E = \left\{
\begin{array}{l| p{1cm}r}
x_i,x_i',x_i'',\ldots,x_i^{(N-1)}; & &  \forall i=1,\ldots,N \\
y_j,y_j',\ldots,y_j^{(N-1)}; & & \forall j = 1,\ldots, M \\
u_k,u_k',\ldots,u_k^{(N-2)}; & & \forall k = 1,\ldots, R \\
\end{array}
\right\} \]

We can treat the system of ODEs as an ideal of algebraic
relations among a set of indeterminates. Use these relations
to define the associated Jacobian matrix.

This matrix has $N^2 + MN + R(N-1)$
columns, one for each ``variable'' in the ground set, and $(N-1)N + (N-1)M$ rows,
one for each relation. The entries in the matrix are polynomials in
$\rr[x,x',\ldots,x^{(N-1)}, y,y',\ldots,y^{(N-1)},u,u',\ldots,u^{(N-2)}]$.
The final step is Gaussian elimination in the Jacobian matrix.
Unlike the strategy in Section~\ref{sec:finding}, any rational function
is permitted here.

\begin{ex} We approach observability of Example \ref{ex:2comp2}
using the algebraic matroid.
The resulting matroid has rank three, with 23 bases and 14 circuits.
We can sort this list to find the circuits including $x_1$ and $x_2$ while
excluding $x_1'$ and $x_2'$; we find the following circuits:
\[
\begin{array}{lll}
\{x_1, y_1\}, & \{x_2, u_1, y_1, y_1'\}, & \text{ and  }  \{x_1, x_2, u_1, y_1'\}.\\
\end{array}
\]
The third circuit contains both variables, so is not useful
for proving observability; but the first two circuits 
constitute a proof of observability.
\end{ex}

\begin{ex} From the SIR Model in Example \ref{ex:SIR1}, let our ODE system be of the form:
$$ S'= \mu N - \beta S I - \mu S $$
$$ I'= \beta S I - (\mu + \gamma) I $$
$$ R'= \gamma I - \mu R $$
$$ y= k I .$$
Taking derivatives, we have the system of equations:
\begin{multline*}
\langle S' + \mu S + \beta S I - \mu N, \, 
S'' + \mu S' + \beta S I' + \beta S' I, \, \\  
 I' + (\mu+\gamma) I - \beta S I, \, 
I'' + (\mu+\gamma) I' - \beta S' I - \beta S I', \, \\
 R' + \mu R - \gamma I, \, 
R'' + \mu R' - \gamma I', \,
y - k I, \,
y' - k I', \,
y'' - k I'' \rangle.
\end{multline*}
These are polynomials in the polynomial ring 
$\rr(p)[S,I,R,S',I',R',S'',I'',R'',y,y',y'']$.

Using the elimination order specified to calculate a Gr\"obner basis,
we find that there are no polynomials in $y, y', y''$ and $R$ only, so the model is generically unobservable.  More precisely, we find the polynomial $-k R' - k \mu R + \gamma y$, but no polynomial involving $y, y', y''$ and $R$ only.

We use a similar strategy to compute the matroid for the SIR Model.
The ground set of the algebraic
matroid for the observability computation is
\[E = \left\{
\begin{array}{cccc}
S, S', S'', &
I, I', I'', &
R, R', R'', &
y,y',y''
\end{array}
\right\} \]

The matroid has rank three, with 123 bases and 146 circuits.
We can sort this list to find the circuits 
including $S$, $I$, and $R$ while
excluding their derivatives; we find the following circuits for each
variable:
\[
\begin{array}{llllll}
\{S, y, y'\} & \{S, y, y''\} &  \{S, y', y''\} &
& \{I, y\}  &  \{I, y', y''\} \\
\end{array}
\]
Any relation in the first row proves that $S$ is observable; similarly,
any relation in the second row proves that $I$ is observable.
No relation from $R$ exists; an elimination of the original ideal proves that
$R$ has no relations that do not also involve its derivatives.

In the linear matroid of differentials 
this is made more pronounced. In {\tt Macaulay2},
the command {\tt kernel(transpose(jacobian(I)))} 
yields a matrix whose row vectors
correspond to variables. The vectors 
corresponding to $R,R'$, and $R''$ are
nonzero in a coordinate where all other 
variables are zero. Therefore, any relation
including {\em one} of $\{R,R',R''\}$ must include at least two.
\end{ex}


\section{Indistinguishability}

Recall two state space models are \emph{indistinguishable} if for any
choice of parameters in the first model, there is a choice of parameters
in the second model that will yield the same dynamics in both models, and
vice versa.  There have been several definitions and approaches to solve this problem in the literature \cite{Godfrey,Raksanyi,Walter1984,Zhang}. 
Here we approach the problem by looking at the input-output equations of the models
and using computational algebra to check indistinguishability.

To start with, to be indistinguishable, two models must have the same
input and output variables.  Since indistinguishable models give the same
dynamics, the structures of their input-output equations should be the same.
In the case that there is one output variable in both models, there is a single
input-output equation.  To say the input-output equations  have the same structure means that
exactly the same differential monomials appear in both input-output equations.

\begin{rmk}
When there are multiple outputs, there will be multiple input-output equations.
To make a unique choice, one should fix a specific monomial order on the 
polynomial ring $\qq(p)[y,u,y', u', \ldots, y^{(N-1)}, u^{(N-1)}, y^{(N)}]$ and
compare the differential monomials appearing in the reduced Gr\"obner bases
of the corresponding ideals.
\end{rmk}

Supposing that the two models have the same structures as described above,
we can let $c(p)$ and $c'(p')$ denote the corresponding coefficient maps of
the two models, respectively.  Here $c : \Theta \rightarrow \rr^m$ and 
$c':  \Theta' \rightarrow \rr^m$, and the components are ordered so that the
components correspond to each other as coming from the same differential
monomial.   Note that the dimensions of the parameter spaces $\Theta$ and $\Theta'$
might be different. We further assume that both coefficient maps are monic on the same coefficient. Indistinguishability is characterized in terms of
the coefficient maps $c$ and $c'$.

\begin{defn}
Suppose that Model 1 and Model 2 have the same input-output structure.
Let $c:  \Theta \rightarrow \rr^{m}$ and $c':  \Theta^{'} \rightarrow \rr^{m}$ 
be the coefficient maps for Model 1 and Model 2, respectively.  We say that:
\begin{itemize}
\item  Model 1 and Model 2 are \emph{indistinguishable} if
for all  $p' \in \Theta'$, there exists at least one  
$p\in \Theta$ such that $c(p)=c'(p')$, and vice versa;
\item  Model 1 and Model 2 are \emph{generically indistinguishable} if,
for almost all $p' \in \Theta'$, there exists at least one  
$p\in \Theta$ such that $c(p)=c'(p')$, and vice versa;
\item  Model 1 and Model 2 are \emph{generically distinguishable}
if they are not generically indistinguishable.
\end{itemize}
\end{defn}

\begin{rmk}
The definition of indistinguishability is equivalent to saying that
$c(\Theta) =  c'(\Theta')$.  The definition of generic indistinguishability
is equivalent to saying that the symmetric difference of $c(\Theta) \triangle c'(\Theta')$
is a set of measure zero.  The definition of generic distinguishability
is equivalent to the existence of an open subset $U \subseteq \Theta$ such that
for all $p \in U$, there is no $p' \in \Theta'$ such that $c(p) = c'(p')$ or the
symmetric condition for $\Theta'$. 
\end{rmk}

A simple observation on distinguishability is that indistinguishable models
must have the same vanishing ideal on the image of the parametrization.
This is usually easy to check in small to medium sized examples.
Once the same vanishing ideal has been established, 
an approach for checking indistinguishability is
to construct the equation system $c(p) = c'(p')$ and attempt to
``solve'' for one set of parameters in terms of the other, and vice versa, using
Gr\"obner basis calculations.  Once this has been done, one must check the
resulting solutions to determine if they satisfy the necessary inequality
constraints of the parameter spaces $\Theta$ and $\Theta'$.  We note that identifiable models with coefficient maps satisfying the same algebraic dependence relationships can always be solved for one set of parameters in terms of the other, and vice versa, but the parameter constraints must still be checked for indistinguishability to hold.

\begin{ex} Consider the following two models,
each of which has three parameters:
\[
\begin{pmatrix} 
x_1' \\
x_2' \\
x_3' \end{pmatrix} = {\begin{pmatrix} 
-a_{01}-a_{21} & 0 & 0 \\
a_{21} & -a_{32} & 0 \\
0 & a_{32} & 0
\end{pmatrix}} {\begin{pmatrix}
x_1 \\
x_2 \\
x_3 \end{pmatrix} } + {\begin{pmatrix}
u_1 \\
u_2 \\
0 \end{pmatrix}}
\quad \quad  y_1=x_3  \]

\[
\begin{pmatrix} 
x_1' \\
x_2' \\
x_3' \end{pmatrix} = {\begin{pmatrix} 
-b_{21} & 0 & 0 \\
b_{21} & -b_{02}-b_{32} & 0 \\
0 & b_{32} & 0
\end{pmatrix}} {\begin{pmatrix}
x_1 \\
x_2 \\
x_3 \end{pmatrix} } + {\begin{pmatrix}
u_1 \\
u_2 \\
0 \end{pmatrix}}
\quad \quad  y_1=x_3. 
\]
The input-output equations for the models are: 
 \begin{alignat*}{7}
 y_1'''& + (a_{32}+a_{01}+a_{21})&y_1''& +
(a_{01}a_{32}+a_{21}a_{32})& y_1'& =& a_{21}a_{32}u_1& + 
a_{32}u_2'&+(a_{01}a_{32} + a_{21}a_{32})&u_2, \\
 y_1'''& + (b_{21}+b_{02}+b_{32})& y_1''& +  (b_{21}b_{02}+b_{21}b_{32})&y_1'& =& b_{21}b_{32}u_1& + b_{32}u_2'& +(b_{32}b_{21})&u_2.
\end{alignat*}
respectively.  
The corresponding coefficient maps are
\[\begin{array}{lclrrrr}
c(a_{01}, a_{21}, a_{32}) &  = &
( a_{32}+a_{01}+a_{21}, & a_{01}a_{32}+a_{21}a_{32}, &
a_{21}a_{32}, &a_{32}, &  a_{01}a_{32} + a_{21}a_{32}), \\[2mm]
c'(b_{02}, b_{21}, b_{32}) & = & (b_{21}+b_{02}+b_{32}, &
 b_{21}b_{02}+b_{21}b_{32}, &  b_{21}b_{32}, &
b_{32}, &b_{32}b_{21}).
\end{array}
\]
The vanishing ideal for model $1$ in the polynomial ring $\qq[c_1, c_2, c_3, c_4, c_5]$
is
\[
\langle c_2 - c_5, \: \: c_1c_4 - c_4^2 - c_5 \rangle
\]
whereas the vanishing ideal for model $2$ is
\[
\langle c_3 - c_5, \: \: c_2c_4^2 - c_1c_4c_5 + c_5^2  \rangle.
\]
Since the two vanishing ideals are not equal, the models are generically
distinguishable.
\end{ex}

\begin{ex} Consider the following variation on the previous example,
where we have simply moved an input from compartment $2$ to compartment $3$.
\[
\begin{pmatrix} 
x_1' \\
x_2' \\
x_3' \end{pmatrix} = {\begin{pmatrix} 
-a_{01}-a_{21} & 0 & 0 \\
a_{21} & -a_{32} & 0 \\
0 & a_{32} & 0
\end{pmatrix}} {\begin{pmatrix}
x_1 \\
x_2 \\
x_3 \end{pmatrix} } + {\begin{pmatrix}
u_1 \\
0 \\
u_3 \end{pmatrix}}
\quad \quad  y_1=x_3  \]

\[
\begin{pmatrix} 
x_1' \\
x_2' \\
x_3' \end{pmatrix} = {\begin{pmatrix} 
-b_{21} & 0 & 0 \\
b_{21} & -b_{02}-b_{32} & 0 \\
0 & b_{32} & 0
\end{pmatrix}} {\begin{pmatrix}
x_1 \\
x_2 \\
x_3 \end{pmatrix} } + {\begin{pmatrix}
u_1 \\
0 \\
u_3 \end{pmatrix}}
\quad \quad  y_1=x_3.   \]

The input-output equations for these models are: \begin{small} \medmuskip=1mu
\thinmuskip=1mu
\thickmuskip=1mu
\begin{alignat*}{8}
y_1'''& + (a_{32}+a_{01}+a_{21})y_1''& + (a_{01}a_{32}+a_{21}a_{32})y_1'&= &
a_{21}a_{32}u_1& + u_3''& +(a_{01}+a_{21}+a_{32})u_3'&+(a_{01}a_{32} + a_{21}a_{32})u_3\\ 
y_1'''& + (b_{21}+b_{02}+b_{32})y_1''& + (b_{21}b_{02}+b_{21}b_{32})y_1'&= &
b_{21}b_{32}u_1& + u_3''& +(b_{02}+b_{32}+b_{21})u_3'& +(b_{02}b_{21} + b_{32}b_{21})u_3,
\end{alignat*} \end{small}
respectively.
In both cases, the vanishing ideal of the model coefficients is the ideal
\[
\langle c_2 - c_5, c_1 - c_4 \rangle,
\]
which suggests that the two models might be indistinguishable.
A simple Jacobian calculation shows that the models are locally identifiable,
and hence we can attempt to solve the system $c(p) = c'(p')$ to
test for indistinguishability.
Solving the system of equations:
$$ a_{32}+a_{01}+a_{21} = b_{21}+b_{02}+b_{32} $$
$$ a_{01}a_{32}+a_{21}a_{32} = b_{21}b_{02}+b_{21}b_{32} $$
$$ a_{21}a_{32} = b_{21}b_{32} $$
we obtain the solutions: 
$$ \left\{ a_{21} = b_{32}, a_{01} = b_{02}, a_{32} = b_{21} \right\}$$
$$\left\{ a_{21} = (b_{21}b_{32})/(b_{02} + b_{32}), 
  a_{01} = (b_{02}b_{21})/(b_{02} + b_{32}), a_{32} = b_{02} + b_{32} \right\}$$
Likewise, one can obtain the solutions:
$$ \left\{ b_{21} = a_{32}, b_{02} = a_{01}, b_{32} = a_{21} \right\}$$
$$\left\{ b_{32} = (a_{21}a_{32})/(a_{01} + a_{21}), 
  b_{02} = (a_{01}a_{32})/(a_{01} + a_{21}), b_{21} = a_{01} + a_{21} \right\}$$
The parameter spaces for these models have all parameters positive.
It is easy to see that for any choice of parameters in the first
model, there is a choice of parameters in the second model that
gives the same input-output equation, and vice versa.
Thus these models are indistinguishable.  Note that
there are two solutions because the models are locally but not
globally identifiable.
\end{ex}

\begin{ex} Now consider the following variation of the previous models,
 where we have added an extra leak parameter to each model and removed the inputs:

\[
\begin{pmatrix} 
x_1' \\
x_2' \\
x_3' \end{pmatrix} = {\begin{pmatrix} 
-a_{01}-a_{21} & 0 & 0 \\
a_{21} & -a_{32} & 0 \\
0 & a_{32} & -a_{03}
\end{pmatrix}} {\begin{pmatrix}
x_1 \\
x_2 \\
x_3 \end{pmatrix} } 
\quad \quad  y_1=x_3  \]

\[
\begin{pmatrix} 
x_1' \\
x_2' \\
x_3' \end{pmatrix} = {\begin{pmatrix} 
-b_{21} & 0 & 0 \\
b_{21} & -b_{02}-b_{32} & 0 \\
0 & b_{32} & -b_{03}
\end{pmatrix}} {\begin{pmatrix}
x_1 \\
x_2 \\
x_3 \end{pmatrix} } 
\quad \quad  y_1=x_3.  \]

The input-output equations for these models is:\begin{small} \medmuskip=0mu
\thinmuskip=0mu
\thickmuskip=0mu
\begin{alignat*}{17}
&y_1''' + (a_{01}+a_{21}+a_{32}+a_{03})&y_1'' + & 
(a_{01}a_{32}&+&a_{21}a_{32}&+&a_{32}a_{03}&+&a_{01}a_{03}&+&a_{21}a_{03})&y_1' &+ 
(a_{01}a_{03}a_{32}+a_{03}a_{21}a_{32})& y_1 & =  0 \\
&y_1''' + (b_{21}+b_{02}+b_{32}+b_{03})&y_1'' + & 
(b_{21}b_{02}&+&b_{21}b_{32}&+&b_{02}b_{03}&+&b_{03}b_{21}&+&b_{03}b_{32})&y_1' &+ 
(b_{02}b_{03}b_{21}+b_{03}b_{21}b_{32})& y_1& = 0
\end{alignat*} \end{small}
In both cases, the vanishing ideal of the model coefficients is the zero ideal which suggests that the two models might be indistinguishable.  These models are clearly unidentifiable since there are $3$ coefficients in $4$ unknown parameters.  
Solving the system $c(p)=c'(p')$, we get the following $6$ solutions:
$$\left\{a_{03} = b_{03}, a_{21} = -a_{01} + b_{02} + b_{32}, a_{32} = b_{21}\right\}$$
$$\left\{a_{03} = b_{03},  a_{21} = -a_{01} + b_{21}, a_{32} = b_{02} + b_{32}\right\}$$
$$\left\{a_{03} = b_{21},   a_{21} = -a_{01} + b_{02} + b_{32}, a_{32} = b_{03}\right\}$$
$$\left\{a_{03} = b_{21},   a_{21} = -a_{01} + b_{03}, a_{32} = b_{02} + b_{32}\right\}$$ 
$$\left\{a_{03} = b_{02} + b_{32},   a_{21} = -a_{01} + b_{21}, a_{32} = b_{03}\right\}$$
$$\left\{a_{03} = b_{02} + b_{32},  a_{21} = -a_{01} + b_{03}, a_{32} = b_{21}\right\}$$
when solving for $\left\{a_{01},a_{21},a_{32},a_{03}\right\}$.  A similar result follows when solving for $\left\{b_{21},b_{02},b_{32},b_{03}\right\}$.  Thus the models are indistinguishable.
\end{ex}

\begin{rmk} Note that the vanishing ideals being equal is only a 
necessary condition for indistinguishability but not in general sufficient.  
For example, suppose that we restrict to the parameter space consisting of
positive parameters and consider  the coefficient maps $c(p_1,p_2)=(p_1,p_1+p_2)$ and $c'(p_1',p_2')=(p_1'+p_2',p_2')$.
The images in both cases have zero vanishing ideal.
However, the models are distinguishable since the image of the first coefficient map
is $\{ (c_1, c_2) \in \rr^2 :  c_2 > c_1 > 0 \}$ whereas the image of the second coefficient map
is $\{ (c_1, c_2) \in \rr^2 :  c_1 > c_2 > 0 \}$.  
\end{rmk}

\begin{rmk}
Some authors also consider a one-sided notion of indistinguishability.
In this definition, Model $1$ is indistinguishable from Model $2$ if every for
every choice of parameters in Model $1$, there is a choice of parameters in 
Model $2$ that can produce the same dynamics.  So Model $2$ is a more
expressive class of models.  It is more difficult to check for this type
of indistinguishability because it need not be the case that the input-output 
equations have the same structure, and so we cannot simply check that the
image of the coefficient map of Model $1$ is contained in the image of the coefficient
map of Model $1$.  As a simple example, if Model $1$
has input-output equation $y' + a_1y = 0$, and Model $2$ has input-output
equation $y'' + b_1y' + b_2 y = 0$, clearly Model $1$ is indistinguishable from Model 
$2$, but this is not detectable by comparing the image of the coefficient maps.
\end{rmk}


\section{Further Reading}

We have demonstrated some techniques to test identifiability, observability, and indistinguishability using a differential algebraic approach.  There are several other approaches to investigate these concepts, so we outline a few of these other methods now for the interested reader.  

For linear models, the global identifiability problem can be solved with the transfer function approach \cite{Bellman} and the similarity transformation approach \cite{WalterLinear, Walter1981}.  For nonlinear models, the differential algebra method has been a powerful technique to test for identifiability \cite{Ljung, Ollivier, Saccomani}.  The main advantage of the differential algebra method is that \textit{global} identifiability can be determined.  On the other hand, there are many approaches to test \textit{local} identifiability, including the Taylor series method \cite{Pohjanpalo}, the generating series method \cite{Walter1982} (with implementations involving identifiability \textit{tableaus} \cite{Balsa-Canto} and exact arithmetic rank \cite{Karlsson}), a method based on the implicit function theorem \cite{Wu, Xia}, a test for reaction networks \cite{Craciun, Davidescu}, and a profile likelihood approach \cite{Raue}.  There are special cases where global identifiability can be determined using a nonlinear variation of the similarity transformation approach \cite{Chappell, Denis-Vidal1996,Vajda} and the direct test \cite{Denis-Vidal2000, Denis-Vidal2001}.  These approaches for global and local identifiability are outlined in greater detail and tested on several models in \cite{Chis} and \cite{Raue2014}.  

For linear models, the concept of observability can be tested using a linear algebra test \cite{Kalman}.  These conditions can be translated to conditions on the graph of the linear compartmental models \cite{Godfrey, Zazworsky}.  For nonlinear models, observability can be tested with differential algebra \cite{Glad, Lu} .  Alternatively, the nonlinear problem has been approached analytically in \cite{Hermann}.  To test local algebraic observability, one can use a probabilistic seminumerical method that solves the problem in polynomial time \cite{Sedoglavic}.    

For linear models, the indistinguishability problem has been analyzed using geometrical rules \cite{Godfrey} and a linear algebra test \cite{Zhang}.  For nonlinear models, indistinguishability was introduced in \cite{Sussmann}.  The problem has been extensively studied for certain classes of nonlinear compartmental models in \cite{Chapman1994,Chapman1996,Godfrey1994,Walter1996} and more generally in \cite{Evans2004}.

This paper is concerned with state space models, but identifiability
and related concepts are also explored heavily in other contexts.
Beltrametti and Robbiano \cite{Beltrametti2012} consider the ideal     
$\langle c_1(p)-c_1(p^*),c_2(p)-c_2(p^*),...,c_m(p)-c_m(p^*)\rangle$
for detecting identifiability in the context of the Hough transform.
Other areas include study of identifiability of graphical models \cite{Allman2015,Drton2011,Foygel2012,
Garcia2010, Tian2010} and identifability of phylogenetic models 
\cite{Allman2011, Long2015, Matsen2008, Rhodes2012}.


\section{Appendix: Algebraic Matroids}

We review the basics of general matroid theory here and especially
main results on algebraic matroids.

\begin{defn}
Let $E$ be a finite set and let $\mathcal{I}$ be a collection of subsets
of $E$ satisfying the following three conditions:
\begin{enumerate}
\item  $\emptyset \in \cali$ 
\item  If $X \in \cali$ and $Y \subseteq X$ then $Y \in \cali$, and
\item If $X,Y \in \cali$ with $|X| < |Y|$ then there exists $y \in Y$ such that
$X \cup \{y\} \in \cali$.
\end{enumerate}
The pair $(E, \cali)$ is called a \emph{matroid} and the elements of $\cali$
are called \emph{independent sets}.
\end{defn}

A special instance of matroids are sets $E$ of vectors in a vector space
where the set $\cali$ consists of all linearly independent subsets of $E$.
A matroid that arises in this way is called a \emph{representable} or a
\emph{linear matroid}.
Various other terminology from linear algebra is also applied in
matroid theory.  A maximal cardinality subset of $\cali$ is
called a \emph{basis}.  The subsets of $E$ that are not in $\cali$ are
called \emph{dependent sets}.  A minimal dependent set is called a \emph{circuit}.
Oxley's text \cite{Oxley} is  a standard reference for background on matroids.

The most important type of matroid for us will be the algebraic matroids whose
properties we review here.

\begin{prop}\cite[Thm 6.7.1]{Oxley}   Suppose $\kk$ is an extension field of a field $\ff$ 
and $E$ is a finite subset of $\kk$.  Then the collection $\cali$ of 
subsets of $E$ that are algebraically independent over 
$\ff$ is the set of independent sets of a matroid on $E$.
The resulting matroid is called an \emph{algebraic matroid}.
\end{prop}

\begin{ex} Let $E=\left\{a_{11},a_{22},a_{12},a_{21}\right\}$ and
 $\ff = \rr(c_1(p),c_2(p),c_3(p))$, where 
 $c_1(p)=-a_{11}-a_{22}, c_2(p)=a_{11}a_{22}-a_{12}a_{21}, c_3(p)=-a_{22}$ 
 from the linear 2-compartment model in Example \ref{ex:2comp4}.  
 Then $I=\left\{ \emptyset, \left\{a_{12}\right\}, \left\{a_{21}\right\}\right\}$
 and $C=\left\{\left\{a_{11}\right\}, \left\{a_{22}\right\}, \left\{a_{12}, a_{21}\right\}\right\}$.
\end{ex}

In our problem, we have the mapping 
$p \mapsto (c_1(p), c_2(p), c_3(p))$
and the variety $V$ of interest is the pre-image of a point 
$\hat{c}=(\hat{c_1},\hat{c_2},\hat{c_3})$ under the map $c$.
Note that the map $c$ has a trivial vanishing ideal; the image
of this map is the full $\rr^3$. The point $\hat{c}$ can therefore
be taken to be a generic point of $\rr^3$ by setting 
$\{\hat{c_1},\hat{c_2},\hat{c_3}\}$ to be algebraically independent
over $\rr$.
This means that the only algebraic constraints 
on the $p$-variables come from the equations
 $\left\{ c(p)=\hat{c} \right\}$.  

Our associated ideal is $P=\left\langle 
c_1(p)-\hat{c_1}, c_2(p)-\hat{c_2},
c_3(p)-\hat{c_3} \right\rangle$, which contains polynomials in 
$\rr(\hat{c})[p] = \rr(\hat{c_1},\hat{c_2},
\hat{c_3})[a_{11},a_{22},a_{12},a_{21}]$. The ideal $P$ is prime, as
confirmed by a Gr\"{o}bner basis computation 
at a randomly chosen point; therefore, computation of the 
algebraic matroid modulo $P$ is
well-defined.

\begin{prop}\cite[Prop 6.7.11]{Oxley}  If a matroid $M$ is algebraic 
over a field $\ff$ of characteristic zero, then $M$ is 
linearly representable over $\ff(T)$ for some finite set $T$ of
transcendentals over $\ff$.
\end{prop}
The following proposition follows from \cite[Proposition 2.14]{KRT2013} 
together with the observation that the tangent space of a variety
is the kernel of its Jacobian matrix:

\begin{prop} Let $P=\left\langle f_1,...,f_m \right\rangle$ be a prime ideal contained in $\ff[x_1,...,x_n]$.  Define the Jacobian matrix $J(P)$ as:

$$ \left( \frac{\partial f_i}{\partial x_j} : 1 \leq i \leq m, 1 \leq j \leq n \right) .$$

This matrix, when considered as a matroid with columns as the ground set and linear independence over $\Frac(\ff[x]/P)$ defining independent set $I$ represents the dual matroid to $M(P)$.  The transpose of the matrix spanning the kernel gives the matroid $M(P)$.
\end{prop}

\begin{ex} Let $c$ be the map $p \mapsto (c_1(p), c_2(p), c_3(p))$ from the linear 2-compartment model in Example \ref{ex:2comp4}.  Then the Jacobian $J(c)$ is given by
$$
\begin{pmatrix}
-1 & -1 & 0 & 0 \\
a_{22} & a_{11} & -a_{21} & -a_{12} \\
0 & -1 & 0 & 0
\end{pmatrix}
$$
A basis for the kernel of this matrix is given by $(0, 0, a_{12}, -a_{21})^T$.  Here, linear independence is taken over $\Frac(\rr(\hat{c})[p]/P) \cong \rr(\hat{c})(a_{12},a_{21})$.  Thus, a vector matroid is given by:
$$
\begin{pmatrix}
0 & 0 & a_{12} & -a_{21}
\end{pmatrix}
$$
where the ground set $E=\left\{1,2,3,4\right\}$ and a set of circuits is given by $C=\left\{\left\{1\right\},\left\{2\right\},\left\{3,4\right\}\right\}$.  This implies that $a_{11}$ and $a_{22}$ are each algebraic over $\rr(\hat{c})$, which implies that $a_{11}$ and $a_{22}$ are each locally identifiable.  This also implies that $\left\{a_{12}, a_{21} \right\}$ is algebraically dependent over $\rr(\hat{c})$.

\end{ex}

\section*{Acknowledgments}
Nicolette Meshkat was partially supported by the Clare Boothe Luce Program from the Luce Foundation and by the David and Lucille Packard Foundation.
Zvi Rosen was partially supported by a Math+X Research 
Grant from the Simons Foundation.
Seth Sullivant was partially supported by the David and Lucille Packard Foundation
and by the US National Science Foundation (DMS 0954865 and 1615660).




\end{document}